\documentclass[a4paper,11pt,dvipsnames]{amsart}
\usepackage{xcolor}
\usepackage{graphicx}
\usepackage{amsmath}
\usepackage{amssymb}
\usepackage{mathtools}
\usepackage{microtype}

\allowdisplaybreaks

\theoremstyle{plain}
\newtheorem{theorem}{Theorem}[section]
\newtheorem{proposition}[theorem]{Proposition}
\newtheorem{lemma}[theorem]{Lemma}

\newtheorem{remark}[theorem]{Remark}

\newtheorem{maintheorem}{Theorem}

\theoremstyle{definition}

\newcommand{\C}{\mathbb{C}}
\newcommand{\R}{\mathbb{R}}
\newcommand{\dd}{\,d}
\newcommand{\tr}{\operatorname{tr}}
\newcommand{\Id}{\mathrm{Id}}
\newcommand{\Area}{\operatorname{Area}}
\newcommand{\wind}{\operatorname{wind}}
\numberwithin{equation}{section}

\usepackage{hyperref}
\hypersetup{colorlinks=true,linkcolor=blue,citecolor=blue,urlcolor=blue,
  pdftitle={Constant mean curvature disks with circular boundary}}

\begin{document}

\title{Constant mean curvature disks with circular boundary}

\author{Davi Maximo}
\address{Department of Mathematics, University of Pennsylvania, Philadelphia, PA 19104, USA}

\author{Ivaldo Nunes}
\address{Departamento de Matem\'atica, Universidade Federal do Maranh\~ao, S\~ao Lu\'{\i}s, MA, Brazil}

\date{\today}

\begin{abstract}
We prove that a compact immersed disk of constant mean curvature in
$\R^3$ whose boundary is mapped diffeomorphically onto a round circle is
a planar disk or a spherical cap.
\end{abstract}

\maketitle

\section{Introduction}\label{sec:intro}

A classical theorem of H. Hopf \cite{Hopf} states that a closed immersed
surface of genus zero with constant mean curvature in $\R^3$ is a round
sphere. The corresponding statement for surfaces with boundary, in the
simplest possible case, has remained open for several decades: is a
compact immersed disk of constant mean curvature bounded by a round
circle necessarily a planar disk or a spherical cap? The problem was posed
by R. Gulliver and R. Kusner as Problem~1.4 at the 1984 AMS Summer
Institute \cite{Brothers1986}. It appears as Problem~2.8 in the recent
survey of R. L\'opez \cite{Lopez}, and we refer to Chapter~7 of the
monograph \cite{LopezBook} for its history.

The disk hypothesis is essential: N. Kapouleas \cite{Kapouleas}
constructed compact immersed constant mean curvature surfaces of higher
genus spanning a circle. For disks, rigidity was known under additional
assumptions: stability, by L. Al\'{\i}as, L\'opez and B. Palmer
\cite{AliasLopezPalmer}; graphicality over the plane of the circle
(for another proof, see \cite{LopezSmallCaps}); area at most that of the
large spherical cap with the same nonzero mean curvature and boundary,
by L\'opez and S. Montiel \cite{LopezMontielBoundedArea}; or $|H|r=1$,
where $r$ is the circle's radius, by F. Brito and R. Sa Earp
\cite{BritoEarp}. For $H=0$, the height above the boundary plane is
harmonic with zero boundary values, so the surface is planar. The free
boundary analogue for disks meeting the boundary of a ball orthogonally
was proved by Nitsche \cite{Nitsche}. Under stability, A. Ros and
E. Vergasta \cite{RosVergasta} left a possible genus-one case, subsequently
excluded by the second author \cite{Nunes}.

Throughout, a smooth immersion of a compact surface with boundary is
smooth up to the boundary with injective differential at every point,
including boundary points, and a spherical cap is either of the two
closed regions into which a circle divides a sphere containing it. The
purpose of this paper is to prove:

\begin{samepage}
\begin{maintheorem}\label{thm:A}
Let $M$ be a smooth compact surface diffeomorphic to the closed disk and
let $x\colon M\to\R^3$ be a smooth immersion with constant mean curvature
$H$. If $x|_{\partial M}$ is a diffeomorphism onto a round circle, then
$x$ is an embedding onto a planar disk, if $H=0$, or onto a spherical
cap, if $H\neq0$.
\end{maintheorem}
\end{samepage}
\pagebreak[0] 

While this manuscript was being prepared, we became aware of independent
work of J. M. Espinar \cite{Espinar2026}, who also obtained
Theorem~\ref{thm:A} by a different method based on the Lawson
correspondence.

Normalize the circle to be the unit circle in the plane $\{x_3=0\}$, and
let $N$ be the unit normal along $x$, with the sign fixed in
Section~\ref{sec:prelim}. The classical balancing formula \cite{KKS}
gives
\[
\int_{\partial M}\langle x,N\rangle\dd\ell=-2\pi H
\]
(Lemma~\ref{lem:flux}), so the boundary average of $\langle x,N\rangle$
is $-H$. Spherical caps satisfy the pointwise relation
$\langle x,N\rangle=-H$ along the boundary: they meet the plane of the
circle at a constant angle. Theorem~\ref{thm:A} is a consequence of the
following boundary rigidity statement, which is the main new result of
the paper.

\begin{maintheorem}\label{thm:B}
Under the hypotheses of Theorem~\textup{\ref{thm:A}}, normalized as in
Section~\textup{\ref{sec:prelim}}, $\langle x,N\rangle=-H$ at every point
of $\partial M$. In particular, $\langle N,e_3\rangle$ is constant along
$\partial M$ and the shape operator equals $H\,\Id$ at every boundary
point.
\end{maintheorem}

By Theorem~\ref{thm:B}, the Hopf differential vanishes along the boundary.
Its holomorphicity and Schwarz reflection then imply that the immersion
is totally umbilical. A separate degree argument establishes
embeddedness. We give these arguments in Section~\ref{sec:A}.

The proof of Theorem~\ref{thm:B} relies on the closed form
$\beta=(N+Hx)\times dx$, used by L\'opez--Montiel
\cite{LopezMontielBoundedArea,LopezMontiel} and called the force form by
B. Smyth and G. Tinaglia \cite{SmythTinaglia}. On a disk $\beta$ has a
primitive $P\colon M\to\R^3$, and we
show that $dP(E_1)\times dP(E_2)=(1+H\langle x,N\rangle)(N+Hx)$ for any
positive orthonormal frame. Two integral identities, obtained from
Green's formula and from the closedness of $\beta$, evaluate the integral
of the right hand side to $-\pi(1-H^2)e_3$. By Stokes' theorem, the
horizontal projection $q$ of the boundary curve $\gamma=P|_{\partial M}$,
parametrized by the circular angle $\theta$, satisfies
\[
\int_0^{2\pi}\det(q,q')\dd\theta=2\pi(1-H^2).
\]
The flux formula gives the same value for the energy
$\int_0^{2\pi}|\gamma'|^2\dd\theta$. The isoperimetric inequality of
A. Hurwitz \cite{Hurwitz} then forces the vertical velocity of $\gamma$,
which is $\langle x,N\rangle+H$, to vanish.

\begin{remark}
In the proof of Theorem~\ref{thm:B}, the disk hypothesis is used to
obtain a global primitive of $\beta$. The identities of
Proposition~\ref{prop:identities} and, consequently,
\[
\int_M(1+H\langle x,N\rangle)(N+Hx)\dd A_g=-\pi(1-H^2)e_3
\]
remain valid for compact oriented constant mean curvature surfaces of
arbitrary genus whose boundary is mapped diffeomorphically onto the unit
circle, with the conventions of Section~\ref{sec:prelim}. In positive
genus, however, $\beta$ need not be exact. Applying Stokes' theorem after
cutting the surface introduces additional terms determined by its
periods, so the comparison between the boundary force curve's signed
area and energy no longer follows from these identities alone.
\end{remark}

\begin{remark}
When $|H|=1$ the energy of $\gamma$ vanishes, $\gamma$ is a point, and
Theorem~\ref{thm:B} says that the surface is tangent along its boundary
to the cylinder over the circle. Together with the argument of
Section~\ref{sec:A}, this recovers the hemisphere theorem of
Brito--Sa Earp \cite{BritoEarp}. In this endpoint case, the boundary
relation $\langle x,N\rangle=-H$ also follows directly from the flux
formula and $|\langle x,N\rangle|\leq1$.
\end{remark}

\begin{remark}
Integral identities and the flux formula have previously been used to
establish rigidity under additional assumptions, including graphicality
\cite{LopezSmallCaps} and stability \cite{AliasLopezPalmer}. The new
ingredient here is that the boundary terms of the
two identities of Proposition~\ref{prop:identities} are evaluated
completely by the flux formula, so that the resulting vector identity
can be compared with the energy of the force curve.
\end{remark}

This paper is organized as follows. In Section~\ref{sec:prelim} we fix
conventions and recall the flux formula and the force form. In
Section~\ref{sec:B} we prove Theorem~\ref{thm:B}, and in
Section~\ref{sec:A} we prove Theorem~\ref{thm:A}.

\subsection{AI Statement}
This work grew out of an investigation of a related open problem,
posed as Problem~2.9 in \cite{Lopez}: whether planar disks and spherical caps are the only stable compact
constant mean curvature surfaces bounded by a circle. This is known for
disks \cite{AliasLopezPalmer}, and we were attempting to bound the genus
of a stable surface.
In that investigation we used GPT-6 Astra
as a computational
assistant: at our direction it carried out a series of computations and
compiled a list of integral identities satisfied by compact constant
mean curvature surfaces bounded by a circle, in any genus. The identity
of Proposition~\ref{prop:vector} was among them. We then realized that
when $M$ is a disk, where the force form has a primitive, this identity
could be used to prove Theorem~\ref{thm:B}, and hence
Theorem~\ref{thm:A}. All statements and proofs in this paper, including
those originating in the model's computations, have been verified
independently by the authors, who take full responsibility for its
contents.

\subsection{Acknowledgements}
We thank Marcos Petrucio Cavalcante for many interesting discussions about the spherical cap problem over the years. D. M. and I. N. were partially supported by CNPq, Grant No. 444531/2024-6. I. N. was also partially supported by CNPq, Grants No. 403869/2024-2 and 400078/2025-2.

\section{Preliminaries}\label{sec:prelim}

\subsection{Conventions}
After a rigid motion and a dilation we assume that $x(\partial M)$ is the
unit circle in $\{x_3=0\}$. Let $p\colon S^1\to\partial M$ be
the diffeomorphism:
\[
x(p(\theta))=e_r(\theta)=(\cos\theta,\sin\theta,0),\quad
t=e_\theta(\theta)=(-\sin\theta,\cos\theta,0),
\]
so that $t=\frac{d}{d\theta}x(p(\theta))$.
We write $g$ for the induced metric, $dA_g$ for its area element, and
identify tangent vectors of $M$ with their images under $dx$. Let $\mu$
be the outward unit conormal along $\partial M$. Since $\partial M$ is
connected, $t\times\mu$ agrees with a global unit normal up to a constant
sign, and we fix the unit normal $N$ by
\begin{equation}\label{eq:normal}
N=t\times\mu\quad\text{along }\partial M .
\end{equation}
We orient $M$ by $N$: an orthonormal frame $(E_1,E_2)$ of $TM$ is
positive when $E_1\times E_2=N$. Throughout, $e_r$, $e_\theta$ and $e_3=(0,0,1)$ denote the fixed
vectors of $\R^3$ introduced above, while $(E_1,E_2)$ always denotes a
positive orthonormal frame of $TM$, identified with its image under $dx$. The induced orientation of $\partial M$ is then
$\tau=-t$, since $\mu\times(-t)=N$, so that for every one-form $\omega$
\begin{equation}\label{eq:induced}
\int_{\partial M}\omega=-\int_0^{2\pi}\omega(p'(\theta))\dd\theta .
\end{equation}
The shape operator $A$ is defined by $dN=-dx\circ A$, $H=\frac12\tr A$,
$|A|^2=\tr A^2$, and $\Delta=\operatorname{div}_g\nabla$, so that
$\Delta x=2HN$; a sphere of radius $R$ with inward normal has $H=1/R$.
Along $\partial M$ we set
\begin{equation}\label{eq:boundary}
s=\langle x,N\rangle,\quad n=\langle N,e_3\rangle,\quad
N=s\, e_r+n\, e_3,\quad \mu=N\times t=-n\, e_r+s\, e_3,
\end{equation}
so that $s^2+n^2=1$,
where we used $e_r\times e_\theta=e_3$ and $e_3\times e_\theta=-e_r$.
The function $s=\langle x,N\rangle$ is defined on all of $M$. Note that
$\langle x,\mu\rangle=-n$ on $\partial M$.

For the caps of the sphere of radius $R=\sqrt{1+h^2}$ centred at $he_3$,
\eqref{eq:normal} makes $N$ the inward normal on the cap contained in
$\{x_3\geq0\}$, so $H=1/R$, and the outward normal on the other cap, so
$H=-1/R$; in both cases $s=-H$ and $n=\pm h/R$ along the boundary.

\subsection{The flux formula}
For a smooth map $F\colon M\to\R^3$, $F\times dF$ denotes the
$\R^3$-valued one-form $v\mapsto F\times dF(v)$. Since $d(dF)=0$,
$d(\tfrac12F\times dF)(U,V)=dF(U)\times dF(V)$, and Stokes' theorem with
\eqref{eq:induced} gives
\begin{equation}\label{eq:vector-area}
\int_MdF(E_1)\times dF(E_2)\dd A_g
=-\frac12\int_0^{2\pi}F(p(\theta))\times\frac{d}{d\theta}F(p(\theta))\dd\theta
\end{equation}
for any positive orthonormal frame; the right hand side is invariant
under translations of $F$.

\begin{lemma}[Flux formula]\label{lem:flux}
We have $|H|\leq1$ and
\[
\int_MN\dd A_g=-\pi e_3,\qquad \int_0^{2\pi}s\dd\theta=-2\pi H,\qquad
\int_0^{2\pi}n\, e_r\dd\theta=0 .
\]
\end{lemma}

\begin{proof}
Apply \eqref{eq:vector-area} to $F=x$: since $dx(E_1)\times dx(E_2)=N$,
we get $\int_MN=-\frac12\int_0^{2\pi}e_r\times e_\theta\dd\theta=-\pi e_3$. Integrating
$\Delta x=2HN$, the divergence theorem gives
$\int_{\partial M}\mu\dd\ell=2H\int_MN=-2\pi He_3$, and decomposing $\mu$
by \eqref{eq:boundary} yields the two boundary integrals. Since $|s|\leq1$
and $\partial M$ has length $2\pi$, $|H|\leq1$.
\end{proof}

This is the unit-circle version of the balancing formula of \cite{KKS},
see also \cite[Lemma~1]{LopezMontiel}. Only Stokes' theorem on $M$ is
used, so it holds for immersions of any genus.

\subsection{The force form}
Let $B=A-H\Id$ and $T=N+Hx$, and consider the $\R^3$-valued one-form
\begin{equation}\label{eq:beta}
\beta=T\times dx .
\end{equation}
Since $H$ is constant, $dT=dN+H\,dx=-dx\circ B$, and therefore
\[
d\beta(E_1,E_2)=-(BE_1)\times E_2+(BE_2)\times E_1=-(\tr B)\,N=0 .
\]
Thus $\beta$ is closed. This form occurs in \cite{LopezMontielBoundedArea}
and \cite[\S2]{LopezMontiel}; it is called the force form in
\cite[\S4]{SmythTinaglia}, where it is shown that its periods are the
forces of cycles. Since $M$ is a disk, $\beta$ has a
primitive
\begin{equation}\label{eq:P}
P\colon M\to\R^3,\qquad dP=\beta,
\end{equation}
smooth up to the boundary and unique up to translation (integrate
$\beta$ along the radii of the closed unit disk). For $H=0$, $P$ is the
conjugate minimal surface. We use only the following two properties
of $P$.

\begin{lemma}\label{lem:P}
For any positive orthonormal frame,
\[
dP(E_1)\times dP(E_2)=\langle T,N\rangle\,T=(1+Hs)\,T .
\]
Along $\partial M$, in the parameter $\theta$,
\begin{equation}\label{eq:gamma-prime}
T=(s+H)\, e_r+n\, e_3,\qquad
\frac{d}{d\theta}P(p(\theta))=T\times t=(s+H)\, e_3-n\, e_r .
\end{equation}
\end{lemma}

\begin{proof}
For $a,u,v\in\R^3$ one has $(a\times u)\times(a\times v)=\langle a,u\times v\rangle a$,
which applied to $a=T$, $u=E_1$, $v=E_2$ gives the first claim since
$\langle T,N\rangle=1+Hs$. The second follows from $x=e_r$ on
$\partial M$, \eqref{eq:boundary} and $dP(p')=T\times dx(p')=T\times t$.
\end{proof}

\subsection{Differential identities}
Let $G=\frac12(|x|^2-1)$ and $x^\top=x-sN$, the tangential part of the
position vector. Then, using the Codazzi equation and $dH=0$ for the
third identity,
\begin{equation}\label{eq:identities}
\begin{aligned}
&\nabla G=x^\top,\qquad \Delta G=2+2Hs,\qquad \Delta N=-|A|^2N,\\
&\nabla s=-Ax^\top,\qquad \Delta s=-2H-|A|^2s .
\end{aligned}
\end{equation}
See, e.g., \cite[(2.1)]{Lopez} for the identities involving $\Delta G$
and $\Delta s$. Note that $G=0$ on $\partial M$ and
$\partial_\mu G=\langle x,\mu\rangle=-n$ there.

\section{Proof of Theorem \ref{thm:B}}\label{sec:B}

We start with two integral identities. The first is Green's formula; the
second uses the closedness of $\beta$. Neither uses that $M$ is a disk.

\begin{proposition}\label{prop:identities}
\begin{align}
\int_M(1+Hs)\,x\dd A_g&=H\int_MG\,N\dd A_g ,\label{eq:id1}\\
\int_M(s+HG)\,N\dd A_g&=\pi H\, e_3 .\label{eq:id2}
\end{align}
\end{proposition}

\begin{proof}
Green's second identity for the components of $x$ and the function $G$
reads $\int_M(x\Delta G-G\Delta x)=\int_{\partial M}(x\,\partial_\mu G-G\,\partial_\mu x)$.
By \eqref{eq:identities} the left hand side is $2\int_M(1+Hs)x-2H\int_MGN$,
and since $G=0$ and $\partial_\mu G=-n$ on $\partial M$ the right hand side
is $-\int_0^{2\pi}n\, e_r\dd\theta=0$ by Lemma~\ref{lem:flux}. This
proves \eqref{eq:id1}.

For \eqref{eq:id2} consider the one-form $G\beta$. As $d\beta=0$,
$d(G\beta)=dG\wedge\beta$, and for a positive orthonormal frame
\begin{align*}
(dG\wedge\beta)(E_1,E_2)&=T\times\bigl(dG(E_1)E_2-dG(E_2)E_1\bigr)\\
&=T\times(N\times x^\top)
=H|x^\top|^2N-(1+Hs)\,x^\top,
\end{align*}
where we used $N\times(aE_1+bE_2)=aE_2-bE_1$, $\nabla G=x^\top$, and
$\langle T,x^\top\rangle=H|x^\top|^2$. Since $G$ vanishes on $\partial M$,
Stokes' theorem gives
\[
H\int_M|x^\top|^2N\dd A_g=\int_M(1+Hs)\,x^\top\dd A_g .
\]
Substituting $|x^\top|^2=2G+1-s^2$ on the left, $x^\top=x-sN$ on the
right, and using \eqref{eq:id1},
\[
2H\!\int_MGN+H\!\int_MN-H\!\int_Ms^2N=H\!\int_MGN-\int_MsN-H\!\int_Ms^2N ,
\]
that is, $\int_M(s+HG)N=-H\int_MN=\pi He_3$ by Lemma~\ref{lem:flux}.
\end{proof}

\begin{remark}
Identity \eqref{eq:id2} can also be obtained from Green's formula applied
to $s$ and to the vector field $f=xs-GN$, which satisfies
$(\Delta+|A|^2)f=-2T$ by \eqref{eq:identities}, and $f=se_r$,
$\partial_\mu f=e_3+e_r\,\partial_\mu s$ on $\partial M$.
\end{remark}

\begin{proposition}\label{prop:vector}
$\displaystyle\int_M(1+Hs)\,T\dd A_g=-\pi(1-H^2)\, e_3$.
\end{proposition}

\begin{proof}
Expanding $T=N+Hx$ and using \eqref{eq:id1}, \eqref{eq:id2} and
Lemma~\ref{lem:flux},
\begin{align*}
\int_M(1+Hs)T&=\int_MN+H\int_MsN+H\int_M(1+Hs)x\\
&=\int_MN+H\int_M(s+HG)N=-\pi e_3+\pi H^2e_3 .
\qedhere
\end{align*}
\end{proof}

We now use the primitive $P$ of \eqref{eq:P}. Let $\gamma(\theta)=P(p(\theta))$,
a smooth closed curve in $\R^3$, let $q=(\gamma_1,\gamma_2)$ be its
horizontal projection, and $\det(q,q')=q_1q_2'-q_2q_1'$. By
\eqref{eq:vector-area} applied to $F=P$, Lemma~\ref{lem:P}, and
Proposition~\ref{prop:vector},
\[
-\frac12\int_0^{2\pi}\gamma\times\gamma'\dd\theta=\int_M(1+Hs)T\dd A_g=-\pi(1-H^2)e_3 .
\]
Taking the third component, for which $(\gamma\times\gamma')_3=\det(q,q')$,
\begin{equation}\label{eq:area}
\int_0^{2\pi}\det(q,q')\dd\theta=2\pi(1-H^2).
\end{equation}
On the other hand, by \eqref{eq:gamma-prime} and $s^2+n^2=1$,
$|\gamma'|^2=(s+H)^2+n^2=1+2Hs+H^2$, so by Lemma~\ref{lem:flux}
\begin{equation}\label{eq:energy}
E:=\int_0^{2\pi}|\gamma'|^2\dd\theta=2\pi(1+H^2)-4\pi H^2=2\pi(1-H^2).
\end{equation}

Finally we recall Hurwitz's form of the isoperimetric inequality
\cite{Hurwitz}. If $q\colon\R/2\pi\mathbb Z\to\R^2\cong\C$ is smooth and
$q=\sum_{k\in\mathbb Z}a_ke^{ik\theta}$, then $\det(q,q')=\operatorname{Im}(\bar qq')$
and Parseval's identity gives
\begin{equation}\label{eq:hurwitz}
\int_0^{2\pi}|q'|^2\dd\theta-\int_0^{2\pi}\det(q,q')\dd\theta
=2\pi\sum_{k\in\mathbb Z}(k^2-k)|a_k|^2\geq0 ,
\end{equation}
with equality if and only if $a_k=0$ for $k\notin\{0,1\}$, i.e.\ $q$ is
constant or a positively oriented circle traversed once at constant
speed.

\begin{proof}[End of proof of Theorem \textup{\ref{thm:B}}]
By \eqref{eq:area}, \eqref{eq:hurwitz}, $|q'|^2=|\gamma'|^2-|\gamma_3'|^2$,
and \eqref{eq:energy},
\[
E=\int_0^{2\pi}\det(q,q')\dd\theta\leq\int_0^{2\pi}|q'|^2\dd\theta
\leq\int_0^{2\pi}|\gamma'|^2\dd\theta=E ,
\]
so both inequalities are equalities. Equivalently, using the Fourier
coefficients of $q$,
\[
0=\int_0^{2\pi}(s+H)^2\dd\theta
  +2\pi\sum_{k\in\mathbb Z}k(k-1)|a_k|^2,
\]
where every term is nonnegative. Equality in the second inequality gives
$\int_0^{2\pi}|\gamma_3'|^2=\int_0^{2\pi}(s+H)^2\dd\theta=0$ by
\eqref{eq:gamma-prime}, hence $s=-H$ on $\partial M$. (No division by
$E$ occurred, so $|H|=1$ is included.) Then $n^2=1-H^2$, and $n$ being
continuous on the circle, it is constant. Thus $N=-He_r+ne_3$ along
$\partial M$ with $n$ constant; differentiating in $\theta$ and using
$dN=-dx\circ A$ gives $At=Ht$, and since $A$ is self-adjoint with trace
$2H$, also $A\mu=H\mu$. Hence $A=H\Id$ on $\partial M$.
\end{proof}

\section{Proof of Theorem \ref{thm:A}}\label{sec:A}

If $H=0$, then $x_3$ is harmonic on $M$ and vanishes on $\partial M$, so
$x_3\equiv0$ and $x$ is an immersion into the plane whose boundary is a
diffeomorphism onto the unit circle; Lemma~\ref{lem:degree} below shows
that $x$ is an embedding onto the closed unit disk. Assume from now on
that $H\neq0$.

\subsection{Umbilicity}
In a conformal coordinate $z=u+iv$ with $g=e^{2\lambda}|dz|^2$, the Hopf
differential is $Q=\varphi\,dz^2$ with $\varphi=\langle x_{zz},N\rangle$;
it is a quadratic differential on the interior of $M$, its zeros are the
umbilical points, and the Codazzi equation reads
$\varphi_{\bar z}=\frac12e^{2\lambda}H_z$, so $Q$ is holomorphic when $H$
is constant \cite{Hopf}. By Theorem~\ref{thm:B}, $\varphi=0$ along
$\partial M$.

Near a boundary point we may take a conformal coordinate, smooth up to
the boundary, in which the boundary arc is a segment of $\{v=0\}$ and the
surface lies in $\{v\geq0\}$: let $v$ be the solution of $\Delta_gv=0$
on a half-disk boundary chart with $v=0$ on the boundary arc and smooth
nonnegative, nontrivial data on the interior arc, so that $v>0$ inside,
$v$ is smooth up to the open boundary arc, and $dv\neq0$ there by the
Hopf lemma; a harmonic conjugate $u$ of $v$ on a smaller simply
connected neighbourhood gives the coordinate $u+iv$. In this coordinate
$\varphi$ is holomorphic on a half-disk, continuous up to the diameter
and zero on it; by the Schwarz reflection principle it extends
holomorphically to the full disk and vanishes on the diameter, hence
identically. By the identity theorem on the connected interior of $M$,
$Q\equiv0$, so $A=H\Id$ on $M$.

Therefore $dT=-dx\circ B=0$, $T=N+Hx$ is constant, and $c=x+H^{-1}N$ is
a constant point with $|x-c|=1/|H|$: the image of $x$ lies on the sphere
$S$ of radius $R=1/|H|$ centred at $c$, and $x$ is a local isometry from
$(M,g)$ into $S$.

\subsection{Injectivity}
\begin{lemma}\label{lem:degree}
Let $F\colon M\to\R^2$ be a smooth immersion of a compact connected
oriented surface with boundary such that $F|_{\partial M}$ is a
diffeomorphism onto a Jordan curve $C$. Then $F$ is an embedding onto the
closed region $\overline\Omega$ bounded by $C$.
\end{lemma}

\begin{proof}
The Jacobian of $F$ has a constant sign $\varepsilon$ on $M$. For
$y\notin C$ the set $F^{-1}(y)$ is contained in the interior of $M$ and is
finite, and
$\varepsilon\,\#F^{-1}(y)=\deg(F,y)=\wind(F|_{\partial M},y)$, which
is $0$ for $y\notin\overline\Omega$ and $\pm1$ for $y\in\Omega$. Thus
$F$ maps the interior of $M$ injectively onto $\Omega$ (no interior point
maps to $C$, since $F$ is a local diffeomorphism on the interior and
points outside $\overline\Omega$ have no preimage). Together with the injectivity of $F$ on $\partial M$, $F$ is
an injective immersion of a compact surface, hence an embedding onto
$\overline\Omega$.
\end{proof}

\begin{proof}[End of proof of Theorem \textup{\ref{thm:A}}]
We first show that $x(M)\neq S$. The induced boundary tangent is $-t$ and
the inward conormal is $-\mu$, so the geodesic curvature of $\partial M$
is $k_g=\langle-e_r,-\mu\rangle=-n$ by \eqref{eq:boundary}, with $n$ the
constant of Theorem~\ref{thm:B}. Since $K=H^2$, Gauss--Bonnet gives
\[
\Area_g(M)=\frac{2\pi(1+n)}{H^2}<\frac{4\pi}{H^2}=\Area(S),
\]
because $n^2=1-H^2<1$. If $x(M)=S$, the area formula for the local
isometry $x\colon M\to S$ would give $\Area_g(M)\geq\Area(S)$, a
contradiction.

Let $y_0\in S\setminus x(M)$ and let $\sigma$ be the stereographic
projection from $y_0$. Then $F=\sigma\circ x$ is an immersion into $\R^2$
whose restriction to $\partial M$ is a diffeomorphism onto the round
circle $\sigma(x(\partial M))$. By Lemma~\ref{lem:degree}, $F$ is an
embedding onto the closed disk bounded by that circle, so $x$ is an
embedding onto the spherical cap of $S$ bounded by the unit circle and
not containing $y_0$.
\end{proof}

\bibliographystyle{amsalpha}
\bibliography{bib}

\end{document}